\documentclass[12pt]{amsart}

\usepackage{amscd,amssymb,amsthm,verbatim}

\newcommand{\const}{\operatorname{const.}}

\newcommand{\dvol}{\operatorname{dvol}}

\newcommand{\Hess}{\operatorname{Hess}}

\newcommand{\Image}{\operatorname{Im}}

\newcommand{\Ker}{\operatorname{Ker}}

\newcommand{\R}{{\mathbb R}}

\newcommand{\Z}{{\mathbb Z}}

\numberwithin{equation}{section}
\theoremstyle{plain}

\newtheorem{lemma}[equation]{Lemma}
\newtheorem{theorem}[equation]{Theorem}
\newtheorem{proposition}[equation]{Proposition}

\theoremstyle{definition}
\newtheorem{definition}[equation]{Definition}

\theoremstyle{definition}

\theoremstyle{definition}
\newtheorem{remark}[equation]{Remark}
\def\<{\langle}
\def\>{\rangle}
\def\({\left(}
\def\){\right)}

\begin{document}

\title{An intrinsic parallel transport in Wasserstein space}
\author{John Lott}
\address{Department of Mathematics\\
University of California - Berkeley\\
Berkeley, CA  94720-3840\\ 
USA} \email{lott@berkeley.edu}

\thanks{Research partially supported
by NSF grant DMS-1207654 and a Simons Fellowship}
\date{January 6, 2017}
\subjclass[2000]{}

\begin{abstract}
If $M$ is a smooth compact connected Riemannian manifold, let $P(M)$ 
denote the Wasserstein
space of probability measures on $M$.
We describe a geometric
construction of parallel transport of some tangent cones along 
geodesics in $P(M)$. 
We show that when everything is smooth, the geometric 
parallel transport agrees with earlier formal calculations.
\end{abstract}

\maketitle

\section{Introduction} \label{section1}

Let $M$ be a smooth compact connected Riemannian manifold
without boundary.
The space $P(M)$ of probability measures of $M$ carries a natural metric,
the Wasserstein metric, and acquires the structure of a length space.
There is a close
relation between minimizing geodesics in $P(M)$ and optimal
transport between measures. For more information on this relation,
we refer
to Villani's book \cite{Villani (2009)}.

Otto discovered a formal Riemannian structure on $P(M)$, underlying
the Wasserstein metric
\cite{Otto (2001)}. One can do formal geometric calculations for this
Riemannian structure \cite{Lott (2008)}.
It is an interesting problem to make these formal considerations
into rigorous results in metric geometry.

If $M$ has
nonnegative sectional curvature
then $P(M)$ is a compact length space with nonnegative curvature
in the sense of Alexandrov
\cite[Theorem A.8]{Lott-Villani (2009)},
\cite[Proposition 2.10]{Sturm (2006)}.
Hence one can define the tangent cone $T_{\mu} P(M)$
of $P(M)$ at a measure $\mu \in P(M)$.
If $\mu$ is absolutely continuous with respect to the volume form
$\dvol_M$ then $T_{\mu} P(M)$ is a Hilbert space
\cite[Proposition A.33]{Lott-Villani (2009)}. More generally, 
one can define tangent cones of $P(M)$ without any curvature assumption
on $M$,
using Ohta's $2$-uniform structure on $P(M)$ \cite{Ohta (2009)}.
Gigli showed that
$T_{\mu} P(M)$ is a Hilbert space if and only if $\mu$ is a ``regular''
measure, meaning that it gives zero measure to any hypersurface which,
locally, is the graph of the difference of two convex functions
\cite[Corollary 6.6]{Gigli (2011)}. 
For examples of tangent cones at nonregular measures,
if $S$ is an embedded submanifold of $M$, and $\mu$ is an
absolutely continuous measure on $S$, then $T_{\mu} P(M)$ was computed
in \cite[Theorem 1.1]{Lott (2016)}.

If $\gamma \: : \: [0,1] \rightarrow M$ is a smooth curve in a
Riemannian manifold then one can define the (reverse) parallel transport
along $\gamma$ as a linear isometry from $T_{\gamma(1)}M$ to $T_{\gamma(0)}M$.
If $X$ is a finite-dimensional Alexandrov space then the replacement
of a tangent space is a tangent cone. If one wants to define a
parallel transport along a curve $c : [0,1] \rightarrow X$, as a
map from $T_{c(1)} X$ to $T_{c(0)} X$, then there is the problem that
the tangent cones along $c$ may not look much alike.  
For example, the curve $c$ may pass through various strata of $X$.
One can deal with
this problem by assuming that $c$ is in the interior of a minimizing
geodesic.  In this case, Petrunin proved the tangent cones along
$c$ are mutually isometric, by constructing a parallel transport
map \cite{Petrunin (1998)}. His construction of the
parallel transport map was based on passing to a subsequential limit in an
iterative construction along $c$.  It is not known whether the
ensuing parallel transport is uniquely defined, although this is
irrelevant for Petrunin's result.

In the case of a smooth curve
$c \: : \: [0,1] \rightarrow P^\infty(M)$ in the space of
smooth probability measures, one can do formal
Riemannian geometry calculations on $P^\infty(M)$ to write down an equation
for parallel transport along $c$ \cite[Proposition 3]{Lott (2008)}. It is a
partial differential equation in terms of a family of functions
$\{\eta_t\}_{t \in [0,1]}$. Ambrosio and Gigli noted that there
is a weak version of this partial differential equation
\cite[(5.9)]{Ambrosio-Gigli (2008)}. By a slight extension,
we will define weak solutions to the formal parallel transport
equation; see Definition \ref{definition4.13}.

Petrunin's construction of parallel transport cannot work in full
generality on $P(M)$, since Juillet showed that there is
a minimizing 
Wasserstein geodesic $c$ with the property that the tangent cones
at measures on the interior of $c$ are not all mutually isometric
\cite{Juillet (2011)}.
However one can consider applying the construction on certain 
convex subsets of $P(M)$. We illustrate this in two cases.
The first and easier case is when $c$ is a Wasserstein geodesic
of $\delta$-measures (Proposition \ref{deltaprop}). 
The second case is when $c$ is a Wasserstein
geodesic of absolutely continuous measures, lying in the interior of a
minimizing Wasserstein geodesic, and satisfying a
regularity condition. Suppose that
$\nabla \eta_1 \in T_{c(1)} P(M)$ is an element of the tangent cone at
the endpoint. Here $\nabla \eta_1 \in L^2(TM, dc(1))$ is a square-integrable
gradient vector field on $M$ and $\eta_1$ is in the Sobolev space 
$H^1(M, dc(1))$. For each sufficiently large
integer $Q$, we
construct a triple 
\begin{equation}
(\nabla \eta_Q, \nabla \eta_Q(0), \nabla \eta_Q(1)) \in
L^2([0,1]; L^2(TM, dc(t))) \oplus L^2(TM, dc(0)) \oplus L^2(TM, dc(1))
\end{equation}
with $\nabla \eta_Q(1) = \nabla \eta_1$, 
which represents an approximate parallel
transport along $c$.

\begin{theorem} \label{theorem1.3}
Suppose that $M$ has nonnegative sectional curvature.
A subsequence of 
$\{ (\nabla \eta_Q, \nabla \eta_Q(0), \nabla \eta_Q(1))\}_{Q=1}^\infty$
converges weakly to a weak 
solution $(\nabla \eta_\infty, \nabla \eta_{\infty,0}, 
\nabla \eta_{\infty, 1})$ of the parallel transport equation with
$\nabla \eta_{\infty, 1} = \nabla \eta_1$. 
If $c$ is a smooth geodesic in $P^\infty(M)$,
$\eta_1$ is smooth, and there is a
smooth solution $\eta$ to the parallel transport equation
(\ref{4.6}) with
$\eta(1) = \eta_1$, then
$\lim_{Q \rightarrow \infty}  (\nabla \eta_Q, \nabla \eta_Q(0), \nabla
\eta_Q(1)) =
(\nabla \eta, \nabla \eta(0), \nabla \eta(1))$ in norm.
\end{theorem}

\begin{remark} \label{remark1.4} In the setting of 
Theorem \ref{theorem1.3}, we can say that
$\nabla \eta_{\infty, 0}$ is the parallel transport of $\nabla \eta_1$
along $c$ to $T_{c(0)} P(M)$.
\end{remark}

\begin{remark} \label{remark1.5} We are assuming that $M$ has nonnegative sectional curvature
in order to apply some geometric results from \cite{Petrunin (1998)}.
It is likely that this assumption could be removed.
\end{remark}

\begin{remark} \label{remark1.7} A result related to 
Theorem \ref{theorem1.3} was proven by
Ambrosio and Gigli when $M = \R^n$ \cite[Theorem 5.14]{Ambrosio-Gigli (2008)},
and extended to general $M$ by Gigli \cite[Theorem 4.9]{Gigli (2012)}.
As explained in \cite{Ambrosio-Gigli (2008),Gigli (2012)}, 
the construction of parallel transport there can be considered to be
extrinsic, in that it is based on embedding the (linear) tangent cones
into a Hilbert space and applying projection operators to form the
approximate parallel transports.  Although we instead use Petrunin's intrinsic
construction, there are some similarities between the two constructions; see 
Remark \ref{added}.  We use some techniques from \cite{Ambrosio-Gigli (2008)},
especially the idea of a weak solution to the parallel transport equation.
\end{remark}

\begin{remark} \label{remark1.6} 
Besides its inherent naturality, the intrinsic construction of parallel 
transport given here is likely to allow for extensions. For example,
using the results of \cite{Lott (2016)}, it seems likely that
Petrunin's construction could be extended to define parallel transport
along Wasserstein geodesics
of absolutely continuous measures on submanifolds of $M$. 
In the present paper we have done this
when the submanifolds have dimension zero 
or codimension zero.
\end{remark}

The structure of this paper is as follows.
In Section \ref{section4} 
we discuss weak solutions to the
parallel transport equation. In Section \ref{section5} we prove Theorem 
\ref{theorem1.3}.

I thank Takumi Yokota and Nicola Gigli for references to the literature.

\section{Weak solutions to the parallel transport equation} \label{section4}

Let $M$ be a compact connected Riemannian manifold without boundary.
Put
\begin{equation} \label{4.1}
P^\infty(M) = \{ \rho \: \dvol_M \: : \: \rho \in C^\infty(M),
\rho> 0, \int_M \rho \: \dvol_M \: = \: 1 \}.
\end{equation}
Given $\phi \in C^\infty(M)$, define a vector field $V_\phi$ on
$P^\infty(M)$ by saying that for $F \in C^\infty(P^\infty(M))$,
\begin{equation} \label{4.2}
(V_\phi F)(\rho \dvol_M) \: = \: 
\frac{d}{d\epsilon} \Big|_{\epsilon = 0} 
F \left( \rho \dvol_M \: - \: \epsilon \: \nabla^i ( \rho \nabla_i \phi) \dvol_M \right).
\end{equation}
The map $\phi \rightarrow V_\phi$ passes to an isomorphism
$C^\infty(M)/\R \rightarrow T_{\rho \dvol_M} P^\infty(M)$.
Otto's Riemannian metric on $P^\infty(M)$ is given \cite{Otto (2001)} by
\begin{align} \label{4.3}
\langle V_{\phi_1}, V_{\phi_2} \rangle (\rho \dvol_M) \: & = \:
\int_M \langle \nabla \phi_1, \nabla \phi_2 \rangle \: \rho \: \dvol_M \\
&  = \: - \: 
\int_M  \phi_1 \nabla^i ( \rho \nabla_i \phi_2) \: \dvol_M. \notag
\end{align}
In view of (\ref{4.2}), we write $\delta_{V_{\phi}} \rho \: = \: - \:  \nabla^i ( \rho \nabla_i \phi)$.
Then
\begin{equation} \label{4.4}
\langle V_{\phi_1}, V_{\phi_2} \rangle (\rho \dvol_M) \:  = \:
\int_M  \phi_1 \: \delta_{V_{\phi_2}} \rho \: \dvol_M
\: = \: \int_M  \phi_2 \: \delta_{V_{\phi_1}} \rho \: \dvol_M.
\end{equation}

To write the equation for parallel transport, let
$c \: : \: [0,1] \rightarrow P^\infty(M)$ be a smooth curve. We write
$c(t) \: = \: \mu_t \: = \: \rho(t) \: \dvol_M$ 
and define $\phi(t) \in C^\infty(M)$, up to a constant, 
by $\frac{dc}{dt} \: = \: V_{\phi(t)}$. 
This is the same as saying
\begin{equation} \label{4.5}
\frac{\partial \rho}{\partial t} + 
\nabla^j \left( \rho \nabla_j \phi \right) = 0.
\end{equation}
Let $V_{\eta(t)}$ be a vector field along $c$,
with $\eta(t) \in C^\infty(M)$.
The equation for $V_{\eta}$ to be parallel along $c$
\cite[Proposition 3]{Lott (2008)} is
\begin{equation} \label{4.6} 
\nabla_i \left( \rho \left( \nabla^i \frac{\partial \eta}{\partial t} \: + \:
\nabla_j \phi \: \nabla^i \nabla^j \eta \right) \right) \: = \: 0.
\end{equation}

\begin{lemma} \label{lemma4.7} 
\cite[Lemma 5]{Lott (2008)} If $\eta, \overline{\eta}$ 
are solutions of (\ref{4.6}) then
$\int_M \langle \nabla \eta, \nabla \overline{\eta} \rangle \: d\mu_t$
is constant in $t$.
\end{lemma}

\begin{lemma} \label{lemma4.8}
Given $\eta_1 \in C^\infty(M)$, there is at most one solution of
(\ref{4.6}) with $\eta(1) = \eta_1$, up to time-dependent additive
constants.
\end{lemma}
\begin{proof}
By linearity, it suffices to consider the case when $\eta_1 = 0$.
From Lemma \ref{lemma4.7}, 
$\nabla \eta(t) = 0$ and so $\eta(t)$ is spatially constant.
\end{proof}

For consistency with later notation, we will write
$C^\infty([0,1]; C^\infty(M))$ for $C^\infty([0,1] \times M)$.

\begin{lemma} (c.f. \cite[(5.8)]{Ambrosio-Gigli (2008)}) \label{lemma4.9}
Given $f \in C^\infty([0,1]; C^\infty(M))$, if $\eta$ satisfies (\ref{4.6}) 
then
\begin{equation} \label{4.10}
\frac{d}{dt} \int_M \langle \nabla f, \nabla \eta \rangle \: d\mu_t =
\int_M \langle \nabla \frac{\partial f}{\partial t}, 
\nabla \eta \rangle \: d\mu_t +  
\int_M \Hess_f(\nabla \eta, \nabla \phi) \: d\mu_t.
\end{equation}
\end{lemma}
\begin{proof}
We have
\begin{align} \label{4.11}
\frac{d}{dt} \int_M \langle \nabla f, \nabla \eta \rangle \: d\mu_t = &
\frac{d}{dt} \int_M \langle \nabla f, \nabla \eta \rangle \: \rho \:
\dvol_M \\ = &
\int_M \langle \nabla \frac{\partial f}{\partial t}, 
\nabla \eta \rangle \: \rho \:
\dvol_M + 
\int_M \langle \nabla f, \nabla \frac{\partial \eta}{\partial t} 
\rangle \: \rho \:
\dvol_M + \notag \\
& \int_M \langle \nabla f, \nabla \eta \rangle \: 
\frac{\partial \rho}{\partial t}
 \:
\dvol_M \notag
\end{align}
Then
\begin{align} \label{4.12}
& \frac{d}{dt} \int_M \langle \nabla f, \nabla \eta \rangle \: d\mu_t
- \int_M \langle \nabla \frac{\partial f}{\partial t}, 
\nabla \eta \rangle \: d\mu_t \\
& = \int_M (\nabla_i f) \:  
\left( \nabla^i \frac{\partial \eta}{\partial t} \right) 
\: \rho \:
\dvol_M - 
\int_M (\nabla_i f) \: (\nabla^i \eta) \: 
\nabla^j \left( \rho \nabla_j \phi \right) \:
\dvol_M \notag \\
& = - \int_M  f \: \nabla_i 
\left( \rho \nabla^i \frac{\partial \eta}{\partial t} \right) 
\: \dvol_M - 
\int_M (\nabla_i f) \: (\nabla^i \eta) \: 
\nabla^j \left( \rho \nabla_j \phi \right) \:
\dvol_M \notag \\
& = \int_M f 
\nabla_i \left( \rho 
(\nabla_j \phi) \: (\nabla^i \nabla^j \eta) \right) \:
\dvol_M + 
\int_M \nabla^j ((\nabla_i f) \: (\nabla^i \eta)) \: 
(\nabla_j \phi) \: \rho \: \dvol_M \notag \\
& = - \int_M (\nabla_i f) \: 
(\nabla_j \phi) \: (\nabla^i \nabla^j \eta) \: \rho \:
\dvol_M \notag \\ 
& + \int_M \nabla^j ((\nabla_i f) \:  (\nabla^i \eta)) \: 
(\nabla_j \phi) \: \rho \: \dvol_M \notag \\
& =
\int_M ( \nabla^j \nabla_i f) \:  (\nabla^i \eta) \: 
(\nabla_j \phi) \: \rho \: \dvol_M \notag \\
& =
\int_M \Hess_f(
\nabla \eta,
\nabla \phi) \: d\mu_t. \notag
\end{align}
This proves the lemma.
\end{proof}

We now weaken the regularity assumptions.
Let $P^{ac}(M)$ denote the absolutely continuous
probability measures on $M$ with
full support. Suppose that $c \: : \: [0,1] \rightarrow P^{ac}(M)$ is a
Lipschitz curve whose derivative $c^\prime(t) \in T_{c(t)} P(M)$
exists for almost all $t$. We can write $c^\prime(t) = V_{\phi(t)}$
with $\nabla \phi(t) \in L^2(TM, dc(t))$. By the Lipschitz assumption,
the essential supremum over $t \in [0,1]$ of 
$\| \nabla \phi(t) \|_{L^2(TM, dc(t))}$ is finite. As before, we write
$c(t) = \mu_t$.

\begin{definition} \label{definition4.13}
Let $c \: : \: [0,1] \rightarrow P^{ac}(M)$ be a
Lipschitz curve whose derivative $c^\prime(t) \in T_{c(t)} P(M)$
exists for almost all $t$.
Given $\nabla \eta_0 \in L^2(TM, d\mu_0)$, $\nabla
\eta_1 \in L^2(TM, d\mu_1)$ and
$\nabla \eta \in L^2([0,1]; L^2(TM, d\mu_t))$, we say that
$(\nabla \eta, \nabla \eta_0, \nabla \eta_1)$ 
is a weak solution of the parallel transport equation if
\begin{align} \label{4.14}
& \int_M \langle \nabla f(1), \nabla \eta_1 \rangle \: d\mu_1 -
 \int_M \langle \nabla f(0), \nabla \eta_0 \rangle \: d\mu_0 = \\
& \int_0^1 
\int_M \left( \left\langle \nabla \frac{\partial f}{\partial t}, 
\nabla \eta \right\rangle \: + \:
\Hess_f(\nabla \eta, \nabla \phi) \right) \: d\mu_t\: dt \notag
\end{align}
for all $f \in C^\infty([0,1]; C^\infty(M))$.
\end{definition}

\begin{remark}
In what follows, there would be analogous results if we replaced
$C^\infty([0,1]; C^\infty(M))$ everywhere
by  $C^0([0,1]; C^2(M)) \cap C^1([0,1]; C^1(M))$. 
We will stick with $C^\infty([0,1]; C^\infty(M))$
for concreteness. 
\end{remark}

From Lemma \ref{lemma4.9}, if 
$c$ is a smooth curve in $P^\infty(M)$ and
$\eta \in C^\infty([0,1]; C^\infty(M))$ is a solution
of (\ref{4.6}) then 
$(\nabla \eta, \nabla \eta(0), \nabla \eta(1))$ is a weak solution of the
parallel transport equation. We now prove the converse.

\begin{lemma} \label{lemma4.15}
Suppose that $c$ is a smooth curve in $P^\infty(M)$.
Given $\eta_0, \eta_1 \in C^\infty(M)$ and $\eta \in C^\infty([0,1];
C^\infty(M))$, if $(\nabla \eta, \nabla \eta_0, \nabla \eta_1)$ 
is a weak solution of the parallel transport equation then
$\eta$ satisfies (\ref{4.6}), $\eta(0) = \eta_0$ and $\eta(1) = \eta_1$ 
(modulo constants).
\end{lemma}
\begin{proof}
In this case, 
equation (\ref{4.14}) is equivalent to
\begin{align} \label{4.16}
& \int_M \langle \nabla f(1), \nabla \eta_1 \rangle \: d\mu_1 -
 \int_M \langle \nabla f(0), \nabla \eta_0 \rangle \: d\mu_0 = \\
&  \int_M \langle \nabla f(1), \nabla \eta(1) \rangle \: d\mu_1 -
 \int_M \langle \nabla f(0), \nabla \eta(0) \rangle \: d\mu_0
+ \notag \\
& \int_0^1 \int_M f 
\nabla_i \left(
\nabla^i \frac{\partial \eta}{\partial t} 
+ \nabla_j \phi \nabla^i \nabla^j \eta \right)
\: d\mu_t\: dt. \notag
\end{align}
Taking $f \in C^\infty([0,1]; C^\infty(M))$ with $f(0) = f(1) = 0$,
it follows that (\ref{4.6}) must hold. Then taking all
 $f \in C^\infty([0,1]; C^\infty(M))$, it follows that
$\nabla \eta_0 = \nabla \eta(0)$ and $\nabla \eta_1 = \nabla \eta(1)$.
Hence $\eta(0) = \eta_0$ and
$\eta(1) = \eta_1$ (modulo constants).
\end{proof}

\begin{lemma} \label{lemma4.17}
Suppose that $c$ is a smooth curve in $P^\infty(M)$.
Given $\nabla \eta_0 \in L^2(TM, d\mu_0)$, 
$\nabla \eta_1 \in L^2(TM, d\mu_1)$,
$\nabla \eta \in L^2([0,1]; L^2(TM, d\mu_t))$ and 
$f \in C^\infty([0,1]; C^\infty(M))$,
suppose that
\begin{enumerate}
\item $(\nabla \eta, \nabla \eta_0, \nabla \eta_1)$ 
is a weak solution to the parallel
transport equation,
\item $f$ satisfies (\ref{4.6}), 
\item $\nabla f(1) = \nabla \eta_1$,
\item
\begin{equation} \label{4.18} 
\int_M | \nabla \eta_0 |^2 \: d\mu_0 \le 
\int_M | \nabla \eta_1 |^2 \: d\mu_1
\end{equation}
and
\item
\begin{equation} \label{4.19} 
\int_0^1 \int_M | \nabla \eta |^2 \: d\mu_t \: dt \le 
\int_M | \nabla \eta_1 |^2 \: d\mu_1
\end{equation}
\end{enumerate}
Then $\nabla f(0) = \nabla \eta_0$, and $\nabla f(t) =  \nabla \eta(t)$
for almost all $t$. 
\end{lemma}
\begin{proof}
From (\ref{4.6}) (applied to $f$) and (\ref{4.14}), we have
\begin{equation} \label{4.20}
\int_M \langle \nabla f(0), \nabla \eta_0 \rangle \: d\mu_0 \: = \:
\int_M \langle \nabla f(1), \nabla \eta_1 \rangle \: d\mu_1 =
\int_M \langle \nabla \eta_1, \nabla \eta_1 \rangle \: d\mu_1.
\end{equation}
From Lemma \ref{lemma4.7},
\begin{equation} \label{4.21}
\int_M \langle \nabla f(0), \nabla f(0) \rangle \: d\mu_0 \: = \:
\int_M \langle \nabla f(1), \nabla f(1) \rangle \: d\mu_1 \: = \:
\int_M \langle \nabla \eta_1, \nabla \eta_1 \rangle \: d\mu_1.
\end{equation}
Then
\begin{equation} \label{4.22}
\int_M | \nabla (\eta_0 - f(0)) |^2 \: d\mu_0 \: = \:
\int_M |\nabla \eta_0|^2 \: d\mu_0 - \int_M |\nabla \eta_1|^2 \: d\mu_1 
\le 0.
\end{equation}
Thus $\nabla f(0) = \nabla \eta_0$ in $L^2(TM, d\mu_0)$.

Next, replacing $f$ by $tf$ in (\ref{4.14}) gives
\begin{equation} \label{4.23}
\int_0^1 \int_M 
\langle \nabla f, \nabla \eta \rangle \: d\mu_t \: dt \: = \:
\int_M \langle \nabla f(1), \nabla \eta_1 \rangle \: d\mu_1 \: = \:
\int_M \langle \nabla \eta_1, \nabla \eta_1 \rangle \: d\mu_1. 
\end{equation}
Then
\begin{align} \label{4.24}
& \int_0^1 \int_M |\nabla f - \nabla \eta|^2 \: d\mu_t \: dt \: = \\
& \int_0^1 \int_M |\nabla f|^2 \: d\mu_t \: dt \: - \: 2
\int_0^1 \int_M \langle \nabla f, \nabla \eta \rangle \: d\mu_t \: dt \: + \:
\int_0^1 \int_M |\nabla \eta|^2 \: d\mu_t \: dt \: =  \notag \\
& \int_M |\nabla f(1)|^2 \: d\mu_1 \: - \: 2
\int_M | \nabla \eta_1|^2 \: d\mu_1 \: + \:
\int_0^1 \int_M |\nabla \eta|^2 \: d\mu_t \: dt \: =  \notag \\
& \int_0^1 \int_M |\nabla \eta|^2 \: d\mu_t \: dt \: - \:
\int_M |\nabla \eta_1|^2 \: d\mu_1 \: \le \: 0. \notag
\end{align}
Thus $\nabla f(t) = \nabla \eta(t)$ in $L^2(TM, d\mu_t)$, for almost all $t$.
\end{proof}

\section{Parallel transport along Wasserstein geodesics} \label{section5}

\subsection{Parallel transport in a finite-dimensional Alexandrov space} 
\label{subsection5.1}

We recall the construction of parallel transport in a finite-dimensional
Alexandrov space $X$.
  
Let $c : [0,1] \rightarrow X$ be a geodesic segment that lies in the
interior of a minimizing geodesic. Then $T_{c(t)} X$ is an isometric
product of $\R$ with the normal cone $N_{c(t)} X$.
We want to construct a parallel transport map from
$N_{c(1)} X$ to $N_{c(0)} X$.

Given $Q \in \Z^+$ and $0 \le i \le Q-1$, define
${c}_i : [0,1] \rightarrow X$ by
${c}_i(u) = c \left( \frac{i + u}{Q} \right)$. 
We define an approximate parallel transport $P_i : N_{{c}_i(1)} X
\rightarrow N_{{c}_i(0)} X$ as follows.  Given
$v \in N_{{c}_i(1)} X$, let $\gamma : [0, \epsilon] \rightarrow X$
be a minimizing geodesic segment with $\gamma(0) = {c}_i(1)$ and
$\gamma^\prime(0) = v$. For each $s \in (0, \epsilon]$, let
$\mu_s : [0, 1] \rightarrow X$ be a minimizing geodesic with
$\mu_s(0) = {c}_i(0)$ and $\mu_s(1) = \gamma(s)$.
Let $w_s \in N_{{c}_i(0)} X$ 
be the normal projection of $\frac{1}{s} \mu_s^\prime(0) \in T_{{c}_i(0)} X$.
After passing to a sequence $s_i \rightarrow 0$, we can assume that
$\lim_{i \rightarrow \infty} w_{s_i} = w \in N_{{c}_i(0)} X$.
Then $P_i(v) = w$. If $X$ has nonnegative Alexandrov curvature then
$|w| \ge |v|$.

In \cite{Petrunin (1998)}, 
the approximate parallel transport from an appropriate dense subset
$L_Q \subset N_{c(1)} X$ to $N_{c(0)} X$ was defined to be
$P_0 \circ P_1 \circ \ldots \circ P_{Q-1}$. It was shown that by
taking $Q \rightarrow \infty$ and applying a diagonal argument,
in the limit one obtains an isometry from a dense subset of  
$N_{c(1)} X$ to $N_{c(0)} X$. This extends by continuity to an
isometry from $N_{c(1)} X$ to $N_{c(0)} X$. 

If $X$ is a smooth Riemannian manifold then $P_i$ is independent of
the choices and can be described as follows. Given
$v \in N_{{c}_i(1)} X$, let $j_v(u)$ be the Jacobi field along
$c$ with $j_v(0) = 0$ and $j_v(1) = v$.  (It is unique since $c$ is in
the interior of a minimizing geodesic.) Then
$P_i(v) = j_v^\prime(0)$.

\subsection{Construction of parallel transport along a 
Wasserstein geodesic of delta measures}
\label{subsection5.1.5}

Let $M$ be a compact connected Riemannian manifold without boundary.  Let
$\gamma : [0, 1] \rightarrow M$ be a geodesic segment that
lies in the interior of a minimizing geodesic. Let
$\Pi : T_{\gamma(1)} M \rightarrow T_{\gamma(0)} M$ be 
(reverse) parallel transport along $\gamma$. Put
$c(t) = \delta_{\gamma(t)} \in P(M)$. Then 
$\{c(t)\}_{t \in [0,1]}$ is a Wasserstein geodesic that
lies in the interior of a minimizing geodesic. 
We apply Petrunin's construction to define parallel transport directly
from the tangent cone $T_{c(1)} P(M)$ to the tangent cone
$T_{c(0)} P(M)$ (instead of the normal cones).
From \cite[Theorem 1.1]{Lott (2016)},
we know that $T_{c(t)}P(M) \cong P_2(T_{\gamma(t)} M)$.

\begin{proposition} \label{deltaprop} The parallel
transport map from $T_{c(1)} P(M) \cong P_2(T_{\gamma(1)} M)$ to
$T_{c(0)} P(M) \cong P_2(T_{\gamma(0)} M)$ is the map
$\mu \rightarrow \Pi_* \mu$.
\end{proposition}
\begin{proof}
Given $Q \in \Z^+$ and $0 \le i \le Q-1$, define
$\gamma_i : [0,1] \rightarrow M$ by 
$\gamma_i(u) = \gamma \left( \frac{i + u}{Q} \right)$ and
${c}_i : [0,1] \rightarrow P(M)$ by
${c}_i(u) = 
\delta_{\gamma_i(u)}$.
We define an approximate parallel transport $P_i : T_{{c}_i(1)} P(M)
\rightarrow T_{{c}_i(0)} P(M)$ as follows.

Given $s \in \R^+$ and a real vector space $V$, let $R_s : V \rightarrow V$
be multiplication by $s$.
Let $\nu$ be a compactly-supported element of $P(T_{\gamma_i(1)} M)$.
For small $\epsilon > 0$, there is a Wasserstein geodesic
$\sigma \: : \: {[0, \epsilon]} \rightarrow P(M)$, with $\sigma(0) = c_i(1)$
and $\sigma^\prime(0)$ corresponding to $\nu \in T_{c_i(1)}PM$, given
by $\sigma(s) = (\exp_{\gamma_i(1)} \circ R_s)_* \nu$. Given
$s \in (0, \epsilon]$, let $\mu_s : [0,1] \rightarrow P(M)$ be a 
minimizing geodesic with $\mu_s(0) = c_i(0) = 
\delta_{\gamma_i(0)}$ 
and $\mu_s(1) = \sigma(s)$. There is a compactly-supported measure 
$\tau_s \in P_2(T_{\gamma_i(0)}M) = T_{c_i(0)} P(M)$ 
so that for $v \in [0,1]$, we have
$\mu_s(v) =  (\exp_{\gamma_i(0)} \circ R_v)_* \tau_s$. If 
$Q$ is large and $\epsilon$ is small then all of the constructions take
place well inside a totally convex ball, so $\tau_s$ is unique and can
be written as $\tau_s = \left( 
\exp_{\gamma_i(0)}^{-1} \circ  \exp_{\gamma_i(1)} \circ R_s \right)_* \nu$.
Then $\lim_{s \rightarrow 0} \frac{1}{s} (\tau_s - \tau_0)$ exists and equals
$(d\exp_{\gamma_i(0)})^{-1}_* \nu$.
Thus $P_i =  (d\exp_{\gamma_i(0)})^{-1}_*$.

Now 
\begin{equation}
P_0 \circ P_1 \circ \ldots \circ P_{Q-1} =
\left( (d\exp_{\gamma_0(0)})^{-1} \circ (d\exp_{\gamma_1(0)})^{-1} 
\circ \ldots \circ (d\exp_{\gamma_{Q-1}(0)})^{-1} \right)_*.
\end{equation}
Taking $Q \rightarrow \infty$, this approaches $\Pi_*$.
\end{proof}

\subsection{Construction of parallel transport along a 
Wasserstein geodesic of absolutely continuous measures}
\label{subsection5.2}

Let $M$ be a compact connected boundaryless Riemannian manifold
with nonnegative sectional curvature.
Then $(P(M), W_2)$ has nonnegative Alexandrov curvature.

Let $c : [0,1] \rightarrow P^{ac}(M)$ be a geodesic segment that lies in the
interior of a minimizing geodesic.
Write $c^\prime(t) = V_{\phi(t)}$. Since $\phi(t)$ is defined up to a
constant, it will be convenient to normalize it by 
$\int_M \phi(t) \: d\mu_t = 0$.
We assume that
\begin{equation} \label{5.1}
\sup_{t \in [0,1]} \| \phi(t) \|_{C^2(M)} \: < \: \infty.
\end{equation}
In particular, this is satisfied if $c$ lies in $P^\infty(M)$.

Let $N_{c(t)} P(M)$ denote the
normal cone to $c$ at $c(t)$.
We want to construct a parallel transport map from
$N_{c(1)} P(M)$ to $N_{c(0)} P(M)$.

Given $Q \in \Z^+$ and $0 \le i \le Q-1$, define
${c}_i : [0,1] \rightarrow P(M)$ by
${c}_i(u) = c \left( \frac{i + u}{Q} \right)$.
Correspondingly, write ${\mu}_{i,u} = \mu_{\frac{i + u}{Q}}$.
We define an approximate parallel transport $P_i : N_{{c}_i(1)} P(M)
\rightarrow N_{{c}_i(0)} P(M)$, using Jacobi fields, as follows.

Let us write ${c}_i^\prime(u) = V_{\phi_i(u)}$, 
i.e. $\phi_i(u) = \frac{1}{Q} \phi \left( \frac{i + u}{Q} \right)$.
The curve ${c}_i$ is given by
${c}_i(u) = (F_{i,u})_* {c}_i(0)$, where
$F_{i,u}(x) = \exp_x (u \nabla_x \phi_i(0))$.
That is, for any $f \in C^\infty(M)$,
\begin{equation} \label{5.2}
\int_M f \: d{c}_i(u) = \int_M f(F_{i,u}(x)) \: d\mu_{i,0}(x).
\end{equation}

If $\sigma_i$ is a variation of $\phi_i(0)$, i.e. $\delta 
\phi_i(0) =\sigma_i$, then taking the variation of (\ref{5.2}) gives
\begin{align} \label{5.3}
\int_M f \: d \delta {c}_i(u) & = 
\int_M \langle \nabla f, d\exp_{u \nabla_x \phi_i(0)}
(u \nabla_x \sigma_i) 
\rangle_{F_{i,u}(x)} \: d\mu_{i,0}(x) \\
& = u \int_M \langle \nabla f, W_{\sigma_i}(u) 
\rangle \: d\mu_{i,u}. \notag
\end{align}
Here
\begin{equation} \label{5.4}
(W_{\sigma_i}(u))_y = d \exp_{u \nabla_x \phi_i(0)}
(\nabla_x \sigma_i), 
\end{equation}
with $y = F_{i,u}(x)$.
The corresponding tangent vector at ${c}_i(u)$ is 
represented by
$L_{\sigma_i}(u) = \Pi_{{c}_i(u)} W_{\sigma_i}(u)$, where
$\Pi_{{c}_i(u)}$ is orthogonal projection on
$\overline{\Image \nabla} \subset L^2(TM, d\mu_{i,u})$.
We can think of $J_{\sigma_i}(u) = u L_{\sigma_{i}}(u)$ as a Jacobi 
field along ${c}_i$.
If $v =  J_{\sigma_i}(1) = 
L_{\sigma_i}(1) = \Pi_{{c}_i(1)} W_{\sigma_i}(1)$ then
its approximate parallel transport along ${c}_i$ is represented by
$w = J_{\sigma_i}^\prime(0) = L_{\sigma_i}(0) = \nabla \sigma_i \in 
\overline{\Image \nabla} \subset L^2(TM, d\mu_{i,0})$.

Next, using (\ref{5.4}), for $f \in C^\infty(M)$ we have
\begin{align} \label{5.5}
\frac{d}{du} 
\int_M \langle V_f, L_{\sigma_i} \rangle \: d\mu_{i,u} = &
\frac{d}{du}  
\int_M \langle V_f, W_{\sigma_i} \rangle \: d\mu_{i,u} = 
\frac{d}{du} 
\int_M \langle \nabla f, d\exp_{u \nabla_x \phi_i(0)}
(\nabla_x \sigma_i) 
\rangle_{F_{i,u}(x)} \: d\mu_{i,0}(x) \\
= & 
\int_M \Hess_{F_{i,u}(x)}(f) \left(
d\exp_{u \nabla_x \phi_i(0)}
(\nabla_x \phi_i(0)),
d \exp_{u \nabla_x \phi_i(0)}
(\nabla_x \sigma_i)
\right) \: d{\mu}_{i,0}(x) + \notag \\
& \int_M \langle \nabla f, D_{\partial_u} 
d \exp_{u \nabla_x \phi_i(0)}
(\nabla_x \sigma_i) 
\rangle_{F_{i,u}(x)} \: d{\mu}_{i,0}(x) \notag \\
= & 
\int_M \Hess(f) \left( \nabla \phi_i(u),
W_{\sigma_i}(u)
\right) \: d{\mu}_{i,u} + \notag \\
& \int_M \langle \nabla f, D_{\partial_u} W_{\sigma_i}(u) 
\rangle \: d{\mu}_{i,u}. \notag
\end{align}
Here $\partial_{u}$ is the vector at $F_{i,u}(x)$ given by
\begin{equation} \label{5.6}
\partial_u = \frac{d}{du} F_{i,u}(x) =
d\exp_{u \nabla_x \phi_i(0)} (\nabla_x \phi_i(0)).
\end{equation}
If instead $f \in C^\infty([0,1]; C^\infty(M))$ then
\begin{align} \label{5.7}
\frac{d}{du} \int_M \langle V_f, L_{\sigma_i} \rangle  \: d\mu_{i,u} \: = &
\int_M \left\langle \nabla
\frac{\partial f}{\partial u}, L_{\sigma_i} \right\rangle  \: d\mu_{i,u} \: 
+  \\
& \int_M \Hess(f) \left( \nabla \phi_i(u),
W_{\sigma_i}(u)
\right) \: d{\mu}_{i,u} + \notag \\
& \int_M \langle \nabla f, D_{\partial_u} W_{\sigma_i}(u) 
\rangle \: d{\mu}_{i,u}. \notag
\end{align}

We will need to estimate $\int_M | W_{\sigma_i}(u) - L_{\sigma_i}(u) |^2
\: d{\mu}_{i,u}$. 

\begin{lemma} \label{lemest}
For large $Q$, there is an estimate
\begin{align} \label{5.11}
& \int_M | W_{\sigma_i}(u) - L_{\sigma_i}(u) |^2
\: d{\mu}_{i,u}  \le \\
& \const 
\|\Hess(\phi_i(\cdot)) \|_{L^\infty([0,1] \times M)}^2
\| L_{\sigma_i}(0) \|_{L^2(TM, d{\mu}_{i,0})}^2.\notag
\end{align}
Here, and hereafter, $\const$ denotes a constant that can depend on
the fixed Riemannian manifold $(M, g)$.
\end{lemma}
\begin{proof}
Since $\Pi_{c_i(u)}$ is projection onto
$\overline{\Image(\nabla)} \subset L^2(TM, d\mu_{i,u})$, and 
$\nabla (\sigma_i \circ F_{i,u}^{-1}) \in 
\Image(\nabla)$, we have
\begin{align} \label{5.8}
\int_M | W_{\sigma_i}(u) - L_{\sigma_i}(u) |^2
\: d{\mu}_{i,u}  \le &
\int_M | W_{\sigma_i}(u) - \nabla (\sigma_i \circ F_{i,u}^{-1}) |_g^2
\: d{\mu}_{i,u} \\
= & \int_M | (dF_{i,u})^{-1}_*
 W_{\sigma_i}(u) - \nabla \sigma_i |_{F_{i,u}^* g}^2
\: d{\mu}_{i,0}. \notag
\end{align}
(Compare with \cite[Proposition 4.3]{Ambrosio-Gigli (2008)}.)
Defining $T_{i,t,x} \: : \: T_xM \rightarrow T_xM$ by 
\begin{equation} \label{5.9}
T_{i,t,x}(z) = (dF_{i,u})_*^{-1} \left( d \exp_{u \nabla_x \phi_i(0)}
(z) \right),
\end{equation}
we obtain
\begin{align} \label{5.10}
& \int_M | W_{\sigma_i}(u) - L_{\sigma_i}(u) |^2
\: d{\mu}_{i,u} \le \\
& \left(
\sup_{x \in M} \| dF_{i,u}^* dF_{i,u}(x) \| \cdot
\| T_{i,u,x} - I \|^2 \right) 
\| L_{\sigma_i}(0) \|_{L^2(TM, d{\mu}_{i,0})}^2. \notag
\end{align}
Since $\sup_{t \in [0,1]} \| \nabla \phi(t) \|_{C^0(M)} < \infty$,
if $Q$ is large then $\| \nabla \phi_i(0) \|_{C^0(M)}$
is much smaller than the injectivity radius of $M$.  In particular,
the curve $\{F_{i,u}(x)\}_{u \in [0,1]}$ lies well within a normal ball
around $x$. Now $T_{i,t,x}$ can be estimated in terms of $\Hess(\phi_i)$.
In general, if a function $h$ on a complete Riemannian manifold satisfies
$\Hess(h) = 0$ then the manifold isometrically splits off an $\R$-factor
and the optimal transport path generated by $\nabla h$ is translation
along the $\R$-factor.  In such a case, the analog of $T_{i,t,x}$ is
the identity map.  If $\Hess(h) \neq 0$ then the divergence of
a short optimal transport path from being a 
translation can be estimated in terms
of $\Hess(h)$. Putting in the estimates gives (\ref{5.11}).
\end{proof}

Using Lemma \ref{lemest}, we have
\begin{align} \label{5.12}
& \left| \int_M \Hess(f) \left( \nabla \phi_i(u),
W_{\sigma_i}(u)
\right) \: d{\mu}_{i,u} - \int_M \Hess(f) \left( \nabla \phi_i(u),
L_{\sigma_i}(u)
\right) \: d{\mu}_{i,u} \right| \le \\
& \const \| \Hess(f) \|_{C^0(M)}  
\|\Hess(\phi_i(\cdot)) \|_{L^\infty([0,1] \times M)}
\| \nabla \phi_i(u) \|_{L^2(TM, d{\mu}_{i,0})}
\| L_{\sigma_i}(0) \|_{L^2(TM, d{\mu}_{i,0})}. \notag
\end{align}

Next, given $x \in M$, consider the geodesic
\begin{equation} \label{5.13}
\gamma_{i,x}(u) = F_{i,u}(x).
\end{equation}
Put
\begin{equation} \label{5.14}
j_{\sigma_i,x}(u) = u (W_{\sigma_i}(u))_{\gamma_{i,x}(u)} \in 
T_{\gamma_{i,x}(u)}M.
\end{equation}
Then $j_{\sigma_i,x}$ is a Jacobi field along $\gamma_{i,x}$, with
$j_{\sigma_i, x}(0) = 0$ and $j_{\sigma_i, x}^\prime(0) = 
\nabla_x \sigma_i$. Jacobi field estimates give
\begin{equation} \label{5.15}
\| D_{\partial_u} W_{\sigma_i}(u) \|_{L^2(TM, d\mu_{i,u})} \le
\const \| \nabla \sigma_i \|_{L^2(TM, d\mu_{i,u})} 
\| \nabla \phi_i(\cdot) \|^2_{L^\infty([0,1] \times M)},
\end{equation}
again for $Q$ large.

\begin{lemma} \label{lemma5.16}
Define $A_i : \left( \overline{\Image(\nabla)} \subset L^2(TM, d\mu_{i,0}) 
\right) 
\rightarrow  \left( \overline{\Image(\nabla)} \subset L^2(TM, d\mu_{i,1})
\right)$ by
\begin{equation} \label{5.17}
A_i (\nabla \sigma_i) = L_{\sigma_i}(1).
\end{equation}
Then for large $Q$, the map $A_i$ is invertible for all
$i \in \{0, \ldots, Q-1\}$.
\end{lemma}
\begin{proof}
Define  $B_i : \left( \overline{\Image(\nabla)} \subset L^2(TM, d\mu_{i,1}) 
\right) 
\rightarrow  \left( \overline{\Image(\nabla)} \subset L^2(TM, d\mu_{i,0}) 
\right)$ 
by
\begin{equation} \label{5.18}
B_i(\nabla f) = \nabla (f \circ F_{i,1}).
\end{equation}
Then whenever $\nabla f \in L^2(TM, d\mu_{i,1})$, we have
\begin{equation} \label{5.19}
(A_i B_i)(\nabla f) = A_i(\nabla(f\circ F_{i,1})) = L_{f \circ F_{i,1}}(1),
\end{equation}
so whenever $\nabla f^\prime \in L^2(TM, d\mu_{i,1})$, for large $Q$ we have
\begin{align} \label{5.20}
& 
\langle \nabla f^\prime, (A_i B_i - I)(\nabla f) \rangle_{L^2(TM, d\mu_{i,1})} 
= \\
& \langle \nabla f^\prime, W_{f \circ F_{i,1}}(1) - \nabla f 
\rangle_{L^2(TM, d\mu_{i,1})} \le \notag \\ 
& \const 
\|\Hess(\phi_i(\cdot)) \|_{L^\infty([0,1] \times M)}
\| \nabla f^\prime \|_{ L^2(TM, d\mu_{i,1})} 
\| \nabla f \|_{ L^2(TM, d\mu_{i,1})}. \notag 
\end{align}
Hence $\| A_i B_i - I \| = o(Q)$, so for large $Q$ the map
$A_i B_i$ is invertible and a
right inverse for $A_i$ is given by $B_i (A_i B_i)^{-1}$. This implies that
$A_i$ is surjective.

Now suppose that $\nabla \sigma \in \Ker(A_i)$ is nonzero,
with $\sigma \in H^1(M, d\mu_{i,0})$.  After normalizing,
we may assume that $\nabla \sigma$ has unit length. Then
\begin{align} \label{5.21}
0 = & \langle \nabla (\sigma \circ F_{i,1}), A_i(\nabla \sigma) 
\rangle_{L^2(TM, d\mu_{i,1})} =
\langle \nabla (\sigma \circ F_{i,1}), L_{\sigma}(1) 
\rangle_{L^2(TM, d\mu_{i,1})} \\
= & \langle \nabla (\sigma \circ F_{i,1}), W_{\sigma}(1) 
\rangle_{L^2(TM, d\mu_{i,1})} =
\langle \nabla \sigma, (dF_{i,1})^{-1} W_{\sigma}(1) 
\rangle_{L^2(TM, d\mu_{i,0})} \notag \\
= &
1 - \langle \nabla \sigma, \nabla \sigma - (dF_{i,1})^{-1} W_{\sigma}(1) 
\rangle_{L^2(TM, d\mu_{i,0})}
\ge 1 - \const \|\Hess(\phi_i(\cdot)) \|_{L^\infty([0,1] \times M)}, \notag
\end{align}
for large $Q$. If $Q$ is sufficiently 
large then this is a contradiction, so $A_i$ is injective.
\end{proof}

Fix ${\mathcal V}_1 \in N_{c(1)} P(M)$. If ${\mathcal V}_1 \neq 0$ then
after normalizing, 
we may assume that it has unit length.
For $Q \in  \Z^+$ large and $t \in [0,1]$,
define ${\mathcal V}_Q(t) \in N_{c(t)} P(M)$ as follows.
First, using Lemma \ref{lemma5.16},
find $\sigma_{Q-1}$ so that ${\mathcal V}_1 = L_{\sigma_{Q-1}}(1)$.
For $t \in \left[ \frac{Q-1}{Q}, 1 \right]$, put
\begin{equation} \label{5.22}
{\mathcal V}_Q(t) = L_{\sigma_{Q-1}}(Qt - (Q-1)).
\end{equation} 
Doing backward recursion, starting with $i = Q-2$, using Lemma \ref{lemma5.16}
we find $\sigma_i$ so that $L_{\sigma_i}(1) = L_{\sigma_{i+1}}(0)
= \nabla \sigma_{i+1}$. For
$t \in \left[ \frac{i}{Q}, \frac{i+1}{Q} \right]$, put
\begin{equation} \label{5.23}
{\mathcal V}_Q(t) =  L_{\sigma_i}(Qt-i).
\end{equation}
Decrease $i$ by one and repeat. The last step is when $i = 0$.

From the argument in  \cite[Lemma 1.8]{Petrunin (1998)},
\begin{equation} \label{5.24}
\lim_{Q \rightarrow \infty}
\sup_{t \in [0,1]} | \|{\mathcal V}_Q(t)\| - 1 | = 0.
\end{equation} 
We note that the proof of \cite[Lemma 1.8]{Petrunin (1998)} only uses
results about geodesics in Alexandrov spaces, it so applies to our
infinite-dimensional setting.  It also uses the assumption that
$c$ lies in the interior of a minimizing geodesic.
After passing
to a subsequence, we can assume that
\begin{equation} \label{5.25}
\lim_{Q \rightarrow \infty}  
\left( {\mathcal V}_Q,  {\mathcal V}_Q(0),
{\mathcal V}_Q(1) \right) =   
\left( {\mathcal V}_\infty,  {\mathcal V}_{\infty,0},
{\mathcal V}_{\infty,1} \right)
\end{equation}
in the weak topology on  
$L^2([0,1]; L^2(TM, d\mu_t)) \oplus L^2(TM, d\mu_0) \oplus L^2(TM, d\mu_1)$.
Note that ${\mathcal V}_{\infty,1} = {\mathcal V}_{1}$.

From (\ref{5.7}), (\ref{5.12}) and (\ref{5.15}), for a fixed 
$f \in C^\infty([0,1]; C^\infty(M))$, on each interval
$\left[ \frac{i}{Q}, \frac{i+1}{Q} \right]$ we have
\begin{align} \label{5.27}
\frac{d}{dt} \int_M \left\langle V_f, {\mathcal V}_Q \right\rangle 
\: d\mu_t = & \int_M
 \left\langle \nabla \frac{\partial f}{\partial t}, 
{\mathcal V}_Q(t) \right\rangle \: d\mu_t \: + \\ 
& \int_M \Hess(f)(\nabla \phi(t), {\mathcal V}_Q(t)) \: d\mu_t + o(Q). \notag
\end{align}
It follows that
$\left( {\mathcal V}_\infty, {\mathcal V}_{\infty, 0}, 
{\mathcal V}_{\infty,1} \right)$ is
a weak solution of the parallel transport equation.
As the limiting vector fields are gradient vector fields, we can write
$\left( {\mathcal V}_\infty, {\mathcal V}_{\infty, 0}, 
{\mathcal V}_{\infty,1} \right) = 
\left( \nabla \eta_{\infty}, \nabla \eta_{\infty, 0}, 
\nabla \eta_{\infty,1} \right)$ for some
$\left(\eta_{\infty}, \eta_{\infty, 0}, 
\eta_{\infty,1} \right) \in 
L^2([0,1]; H^1(M, d\mu_t)) \oplus H^1(M, d\mu_0) \oplus H^1(M, d\mu_1))$.

Suppose that $c$ is a smooth geodesic in $P^\infty(M)$,
that ${\mathcal V}_1$ (and hence
$\eta_{\infty, 1}$) is smooth and that
there is a smooth solution $\eta$ to the parallel transport
equation (\ref{4.6}) with $\nabla \eta(1) = \nabla \eta_{\infty, 1}$. 
By Lemma \ref{lemma4.7}, $\| \nabla \eta(t)\|$ is independent of $t$.
By Lemma \ref{lemma4.17},
$\left( \nabla  \eta_{\infty}, \nabla \eta_{\infty, 0}, 
\nabla \eta_{\infty,1} \right) =
\left( \nabla \eta, \nabla \eta(0), \nabla \eta(1) \right)$.
We claim that
\begin{equation} \label{5.28}
\lim_{Q \rightarrow \infty}  (\nabla \eta_Q,  \nabla \eta_Q(0),
\nabla \eta_Q(1)) =   
\left( \nabla \eta, \nabla \eta(0), \nabla \eta_{\infty,1} \right)
\end{equation}
in the norm topology on $L^2([0,1]; L^2(TM, d\mu_t)) \oplus 
L^2(TM, d\mu_0) \oplus L^2(TM, d\mu_1)$.
This is because of the general fact that if
$\{x_i\}_{i=1}^\infty$ is a sequence in a Hilbert space $H$
with $\lim_{i \rightarrow \infty} |x_i| = 1$, and
there is some unit vector $x_\infty \in H$ so that every weakly convergent
subsequence of $\{x_i\}_{i=1}^\infty$ has weak limit $x_\infty$,
then $\lim_{i \rightarrow \infty} x_i = x_\infty$ in the norm
topology.

In particular, 
\begin{equation} \label{5.29}
\lim_{Q \rightarrow \infty} \nabla \eta_Q(0) = \nabla \eta(0)
\end{equation}
in the norm topology on $L^2(TM, d\mu_0)$.

This proves Theorem \ref{theorem1.3}.

\begin{remark} \label{added}
The construction of parallel transport in
\cite[Section 5]{Ambrosio-Gigli (2008)} and \cite[Section 4]{Gigli (2012)}
is also by taking the limit of an iterative
procedure.  The underlying logic in 
\cite{Ambrosio-Gigli (2008),Gigli (2012)} is 
different than what we use, which results in a different algorithm.
The iterative 
construction in \cite{Ambrosio-Gigli (2008),Gigli (2012)}
amounts to going forward along the curve $c$ applying certain
maps ${\mathcal P}_i$, instead of going backward along $c$ using the
inverses of the $A_i$'s as we do.  In the case of $\R^n$, the map
${\mathcal P}_i$ is the same as 
$A_i$, but this is not the case in general. The map
${\mathcal P}_i$ is nonexpanding, which helps the construction in
\cite{Ambrosio-Gigli (2008),Gigli (2012)}.
In contrast, $A_i^{-1}$ is not nonexpanding. In order to
control its products, we use the result
(\ref{5.24}) from \cite{Petrunin (1998)}.
\end{remark}

\end{document}